\documentclass[11pt,reqno,oneside]{amsart}

\UseRawInputEncoding 

\usepackage{eucal,color,graphicx,subfigure,mathrsfs}

\usepackage[a4paper, total={6in, 9.1in}]{geometry}

\usepackage{hyperref}
\usepackage{xcolor}
\hypersetup{
  colorlinks,
  linkcolor={blue!80!black},
  urlcolor={blue!80!black},
  citecolor={blue!80!black}
}
\usepackage[capitalize,noabbrev]{cleveref}

\numberwithin{equation}{section}

\newtheorem{theorem}{Theorem}[section]
\newtheorem{lemma}[theorem]{Lemma}
\newtheorem{proposition}[theorem]{Proposition}

\theoremstyle{definition}

\theoremstyle{remark}
\newtheorem{remark}[theorem]{Remark}

\numberwithin{equation}{section}

\title[Boundary homogenization of a class of obstacle problems]{ Boundary homogenization of a class of obstacle problems}

\author{Jingzhi Li}
\address{Department of Mathematics, Southern University of Science and Technology, Shenzhen 518055, P.~R.~China.}
\email{li.jz@sustech.edu.cn}

\author{Hongyu Liu}
\address{Department of Mathematics, City University of Hong Kong, Kowloon, Hong Kong SAR, P.~R.~China.}
\email{hongyu.liuip@gmail.com; hongyliu@cityu.edu.hk}

\author{Lan Tang}
\address{School of Mathematics and Statistics, Central China Normal University, Wuhan,
Hubei 430079,  P.~R.~China}
\email{lantang@mail.ccnu.edu.cn}

\author{Jiangwen Wang}
\address{School of Mathematics and Statistics, Central China Normal University, Wuhan,
Hubei 430079,  P.~R.~China}
\email{jiangwen\_wang@126.com}

\begin{document}

\begin{abstract}

We study homogenization of a boundary obstacle problem on $ C^{1,\alpha} $ domain $D$ for some elliptic equations with uniformly elliptic coefficient matrices $\gamma$.  For any $ \epsilon\in\mathbb{R}_+$,  $\partial D=\Gamma \cup \Sigma$, $\Gamma \cap \Sigma=\emptyset $ and $ S_{\epsilon}\subset \Sigma $ with suitable assumptions,\ we prove that as $\epsilon$ tends to zero, the energy minimizer $ u^{\epsilon} $ of $ \int_{D} |\gamma\nabla u|^{2} dx $, subject to $ u\geq \varphi $ on $ S_{\varepsilon} $, up to a subsequence, converges weakly in $ H^{1}(D) $ to  $ \widetilde{u} $  which minimizes the energy functional $\int_{D}|\gamma\nabla u|^{2}+\int_{\Sigma} (u-\varphi)^{2}_{-}\mu(x) dS_{x}$, where $\mu(x)$ depends on the structure of $S_{\epsilon}$ and $ \varphi $ is any given function in $C^{\infty}(\overline{D})$.

\medskip

\noindent{\bf Keywords:}~~Homogenization; boundary obstacle; correctors; asymptotic analysis

\medskip

\noindent{\bf 2010 Mathematics Subject Classification:}~~35B27; 35B40

\end{abstract}



\maketitle

\section{Introduction}

Let $ D\subset \mathbb {R}^{n}$ ($n>2$) be a bounded open subset whose boundary satisfies
\begin{eqnarray}
\label{Prob1}
\partial D=\Gamma \cup \Sigma, \  \Gamma \cap \Sigma=\emptyset \ \ \text{and} \ \ \partial D\in C^{1,\alpha}
\end{eqnarray}
 for some constant $\alpha \in (0,1) $. For any $ \epsilon\in\mathbb{R}_+$, we let $ S_{\epsilon} $ be a subset of $ \Sigma $ with some special structure, which will be precisely stated later. Throughout, we assume:
 \begin{enumerate}
 \item[(a1)]~$ \gamma(x) = (\gamma_{ij}(x))_{n \times n}$ is an $n\times n$ symmetric matrix-valued function on $D$ and there exist two positive constants $a\leq b$ such that 
 $$ a I \leq \ \gamma(x) \ \leq \ b I, \ \  \ x\in D, $$
where $I$ is the identity matrix; 

 \item[(a2)]~$\phi$ and $\psi$ are both smooth functions defined on $\overline{D}$.
 \end{enumerate}

Consider the following variational problem
\begin{eqnarray}
\label{Prob2}
\inf_{v\in K}J(v)
\end{eqnarray}
where
\begin{equation}\label{eq:functional1}
J(v):=\int_{D}|\gamma\nabla v|^{2}dx \ \text{and} \ K:=\{v\in H^{1}(D):v|_{\Gamma}=\psi \ \text{and}\ v|_{S_{\epsilon}}\geq \varphi\} .
\end{equation}
 Let $ u^{\epsilon}$ be the solution to the above problem. In this paper, our main purpose is to study the asymptotic behavior of $ u^{\epsilon} $ when $ \epsilon\rightarrow0 $ under suitable assumptions on $S_{\epsilon}$ and $\Sigma$. That is, we are concerned with the boundary homogenization associated with the variational problem \eqref{Prob2}. This is an important problem with a strong practical background. In fact, it can be used to describe the mathematical model  for semipermeable membranes, where the function $ \varphi(x) $ signifies an external pressure, and the set $S_{\epsilon} $ is considered as a subset of the boundary composed of the part through which the liquid passes on the semipermeable membrane; see \cite{dl76} for more relevant details.

If $ \gamma(x) $ is the identity matrix, there is a long history to study (\ref{Prob2}) with rich results in the literature. When the set $S_{\epsilon}$ lies inside $D$, the problem can be viewed as the homogenization of a variational problem on a perforated domain. For this problem,  \cite{cm82a, cm82b} firstly considered the periodic homogenization and established that the limiting energy functional contains a strange term which depends on the capacity of $S_{\epsilon}$. Later \cite{cm07} obtained the stochastic homogenization result under the setting of the stationary ergodic case. For more general energy structure, see \cite{NA02} and  \cite{tang12} and the references cited therein. When $ S_{\epsilon} \subset \partial D$, the problem becomes more tricky. If $S_{\epsilon}$ lies on the straight part of the boundary with some suitable assumptions, Caffarelli and Mellet {\cite{cm08}} established the random homogenization result of (\ref{Prob2}) under the stationary ergodic setting. If $ S_{\epsilon} $ and $ \partial D $ satisfy more general conditions, Yang {\cite{yang10}} established the homogenization result, which contains an unusual term in the limiting functional.

In this paper, we considered the general case where $\gamma$ is variable matrix-valued function as prescribed in (a1), which is equivalent to introducing a Riemannian metric to the variational problem \eqref{Prob1}--\eqref{eq:functional1} in the domain $D$. Before stating the main result, we introduce more assumptions on the structure $  S_{\epsilon} $ in our study. Firstly for any $ \epsilon >0 $, we define
\begin{equation}\label{eq:lo1}
 T_{\epsilon} =\big(\bigcup_{k} B_{r_{\epsilon,k}}(x_{\epsilon,k})\big) \cap D \ \ \ \text{and}\ \ \ S_{\epsilon} =\big(\bigcup_{k} B_{r_{\epsilon ,k}}(x_{\epsilon,k})\big) \cap \Sigma
\end{equation}
where $ x_{\epsilon , k} \in \Sigma $. It is assumed that for any $ k, \widetilde{k} \in \mathbb{Z}^{n}( k\neq \widetilde{k}) $, there holds
$$ |x_{\epsilon,k}-x_{\epsilon,\widetilde{k}}| \geq  2\epsilon,   $$
and there exist two constants $ c_{1},c_{2} $ (independent of $ k $ and $ \epsilon $) such that
$$ r_{\epsilon,k} = \widetilde{r_{\epsilon ,k}}\epsilon^{\frac{n - 1}{n - 2}}      $$
where $ c_{1} \leq \widetilde{r_{\epsilon , k }} \leq c_{2} $. Clearly, the number of such balls is $ O(\epsilon ^{1-n}) $.

Secondly, we will construct a class of density functions $ \mu_{\epsilon}(x) $ on $ D $ as follows :
\begin{equation}\label{eq:df1}
\mu_{\epsilon}(x) : = \sum_{k}\frac{\big(\sum_{i=1}^{\infty}\widetilde{C_{i}}\epsilon^{i-1}\big)\widetilde{r_{\epsilon,k}}^{n-2}}{\epsilon \widetilde{r_{\epsilon,k}}^{n-2}-1}\nabla\cdot(\gamma^{T}\gamma x)\frac{1}{\epsilon}\chi_{B_{\epsilon}(x_{\epsilon,k})}(x) ,
\end{equation}
where $ \displaystyle\chi_{B_{\epsilon}(x_{\epsilon,k})}(x) $ is the characteristic function of $ B_{\epsilon}(x_{\epsilon,k}) $ and  $ \{ \widetilde{C_{i}}\}_{i=1}^{\infty} $ are bounded. We shall assume that $\mu_\epsilon$ converges properly as $\epsilon\rightarrow 0$ as descried in the following theorem.

The main result of this paper is given as follows.

\begin{theorem}
\label{thm}
Assume that as $ \epsilon \rightarrow 0 $, $ \mu_{\epsilon}(x) dx \rightarrow \mu(x)dS_{x} $ in $ H^{-1}(D) $,  where $ \mu(x) $ is some density function on $ \Sigma $. If $u^{\epsilon}$ is the minimizer of (\ref{Prob2}), then as $ \epsilon \rightarrow 0 $, $ u^{\epsilon}$ converges weakly to $\widetilde{u} $ in $ H^{1}(D) $ where $ \widetilde{u} $ minimizes
$$ J_{\mu}(v)=\int_{D}|\gamma\nabla v|^{2}dx+c_{n}\int_{\Sigma}(v-\varphi)^{2}_{-}\mu(x)dS_{x} $$
over $ \widetilde{K}=\{ v\in H^{1}(D):v|_{\Gamma}=\psi\}$.
\end{theorem}

\begin{remark}
We would like to point out that a similar assumption to the one in Theorem~\ref{thm} was considered in \cite{yang10}. However, the author in \cite{yang10} considered the special case of Theorem~\ref{thm} with $\gamma=I$. The presence of a general $\gamma$ makes the corresponding analysis in proving Theorem~\ref{thm} much more challenging and delicated.
\end{remark}

The remaining part of this  paper is arranged as follows. In Section 2 , we introduce a key lemma for correctors, Proposition \ref{prop}, with which Theorem \ref{thm} can be proved. In Section 3, we construct the corrector and show that it satisfies the conditions (\ref{pr1})-(\ref{pr4}) in Proposition \ref{prop}.  In Section 4, we establish the limiting property of the corrector to finish the proof of Proposition \ref{prop}.

In what follows, $C$ shall signify a generic positive constant which may vary in different inequalities. The symbol $"\rightharpoonup" $ denotes the notion of weak convergence.

\vspace{30pt}

\section{ Proof of the main theorem }

The proof of the Theorem~\ref{thm} critically relies on the following proposition, whose proof is postponed to Section 3-4.

\begin{proposition} \label{prop}
For any $ \epsilon>0 $, let $T_{\epsilon}$ and $S_{\epsilon}$ be defined by (\ref{eq:lo1}), and the density function $ \mu_{\epsilon}(x)$ is given by (\ref{eq:df1}). Assume that as $ \epsilon \rightarrow 0 $, $ \mu_{\epsilon}(x) dx \rightarrow \mu(x)dS_{x} $ in $ H^{-1}(D) $,  where $ \mu(x) $ is some density function on $ \Sigma $. Then there exists a function $ \omega_{\epsilon}(x)\in H^{1}(D) $ satisfying the following conditions:
\begin{eqnarray}
\label{pr1} \omega_{\epsilon}(x)=1 \ \ \ \forall x\in T_{\epsilon} , \\
\label{pr2} \nabla\cdot(\gamma^{T}\gamma\nabla \omega_{\epsilon})\ = \ 0 \ \  \forall  x\in D ,\\
\label{pr3} ||\omega_{\epsilon}||_{L^{\infty}(D)}\ \leq \ C , \ \ \ \   \\
\label{pr4}  \omega_{\epsilon}(x)\rightharpoonup 0 \ \ \ \text{in} \ \  H^{1}(D) \ \text{as}\ \epsilon\rightarrow0.
\end{eqnarray}
Furthermore,  given any function $v_{\epsilon}\in H^1(D)$ satisfying:
$$ v_{\epsilon}(x) \geq 0 \  \ \text{for} \ \  x\in\ T_{\epsilon},  ||v_{\epsilon}||_{L^{\infty}}\leq C\ \text{and}$$
  $$ \  v_{\epsilon} \rightharpoonup v\ \ \text{in}\ \  H^{1}(D)\ \text{as} \ \epsilon \rightarrow 0,$$
then there holds:
\begin{eqnarray}
\label{pr5}
 \lim_{\epsilon \rightarrow 0}\int_{D}(\gamma\nabla \omega_{\epsilon} )\cdot(\gamma\nabla v_{\epsilon}) \phi dx\geq-\int_{\Sigma}v\phi\mu(x)dS_{x}
\end{eqnarray}for any $ \phi(x)\in C^{\infty}(\overline{D}) $ with $ \phi|_{\Gamma}=0 $.
The equality  in (\ref{pr5}) also holds if $ v_{\epsilon} = 0 $ on $ T_{\epsilon} $.
\end{proposition}

From the above proposition, we have the following result :
\begin{lemma}
\label{le2}
For any $ \phi\in\{\phi\in C^{\infty}(\overline{D}):\phi|_{\Gamma}=0\}$, we have that
 $$ \lim_{\epsilon\rightarrow 0}\int_{D}|\gamma\nabla \omega_{\epsilon}|^{2}\phi dx=\int_{\Sigma}\phi\mu(x)dS_{x}. $$
\end{lemma}
\begin{proof} Set $ v_{\epsilon}=1-\omega_{\epsilon} $, then we see obviously that \ $ v_{\epsilon}|_{T_{\epsilon}}=0 $ and $ v_{\epsilon}\rightharpoonup 1 $ in $ H^{1}(D) $. Using Proposition \ref{prop}, one can readily show that
$$ \lim_{\epsilon\rightarrow 0}\int_{D}|{\gamma\nabla \omega_{\epsilon}}|^{2}\phi dx=\int_{\Sigma}\phi\mu(x)dS_{x}. $$
\end{proof}

By Lemma \ref{le2} and Proposition \ref{prop}, we have the following property :
\begin{lemma}
\label{le3}
Assume that $ u^{\epsilon} \rightharpoonup \widetilde{u} $ in $ H^{1}(D) $ , then we have the following lower semi-continuous property :
$$ \liminf_{\epsilon \rightarrow 0}\int_{D}|{\gamma\nabla u^{\epsilon}}|^{2}dx \ \geq \  \int_{D}|{\gamma\nabla \widetilde{u}}|^{2}dx+\int_{\Sigma}(\widetilde{u}-\varphi)^{2}_{-}\mu(x)dS_{x} . $$
\end{lemma}
\begin{proof}
Let us decompose $  u^{\epsilon} = u^{\epsilon}_{+} - u^{\epsilon}_{-}  $. Then we obviously have $ u^{\epsilon}_{+}  \rightharpoonup \widetilde{u}_{+} $ in $ H^{1}(D) $ ($ u^{\epsilon}_{-}  \rightharpoonup \widetilde{u}_{-} $ in $ H^{1}(D)$ respectively ).

For $ u^{\epsilon}_{+}$, we apply the classical lower semicontinuity property :
$$   \liminf_{\epsilon \rightarrow 0}\int_{D}|\gamma \nabla u^{\epsilon}_{+}|^{2} dx \geq \int_{D}|\gamma \nabla \widetilde{u}_{+}|^{2} dx. $$

In order to prove Lemma \ref{le3}, we need to prove the following modified lower semicontinuity property :

$$   \liminf_{\epsilon \rightarrow 0}\int_{D}|\gamma \nabla u^{\epsilon}_{-}|^{2} dx \geq \int_{D}|\gamma \nabla \widetilde{u}_{-}|^{2} dx + \int_{\Sigma}(\widetilde{u} - \varphi)^{2}_{-}\mu(x) dS_{x}. $$
Here we consider the following two estimates :
\begin{eqnarray}
\label{na6}
 \liminf_{\epsilon \rightarrow 0}\frac{1}{2}\int_{\omega_{\epsilon} \leq \theta}|\gamma\nabla u^{\epsilon}_{-}|^{2} dx \geq {\int_{D}(\gamma \nabla\phi)\cdot(\gamma \nabla\widetilde{u}_{-})dx} - \frac{1}{2}\int_{D}|\gamma \nabla\phi|^{2} dx,
\end{eqnarray}
and
\begin{eqnarray}
\label{na7}
 \frac{1}{2}\int_{\omega_{\epsilon} > \theta} |\gamma \nabla u^{\epsilon}_{-}|^{2} dx \geq -\frac{1}{2}\int_{D}|\gamma \nabla \omega_{\epsilon}|^{2}\widetilde{\phi}^{2} dx -{\int_{D}(\gamma \nabla \omega_{\epsilon})\cdot (\gamma \nabla u^{\epsilon}_{-})\widetilde{\phi} dx} +C\theta+C\theta^{\frac{1}{2}},
\end{eqnarray}
where $ \theta $ is a positive (small) constant, and $ \phi $, $\widetilde{\phi}  $ are two test functions (to be determined).

Suppose the above two estimates hold. Then we can use Proposition \ref{prop} and Lemma \ref{le2} to obtain
\begin{align*}
 \liminf_{\epsilon \rightarrow 0}\int_{D}\frac{1}{2}|{\gamma\nabla u^{\epsilon}}|^{2}dx \geq \ & \  \liminf_{\epsilon \rightarrow 0} \left(\frac{1}{2}\int_{\omega_{\epsilon} \leq \theta}|\gamma\nabla u^{\epsilon}_{-}|^{2} dx + \frac{1}{2}\int_{\omega_{\epsilon}> \theta}|\gamma \nabla u^{\epsilon}_{-}|^{2} dx \right) \\
\geq \ & \int_{D}(\gamma \nabla\phi)\cdot(\gamma \nabla\widetilde{u}_{-})dx - \frac{1}{2}\int_{D}|\gamma \nabla\phi|^{2} dx \\
- \ &  \frac{1}{2}\int_{\Sigma}\widetilde{\phi}^{2}\mu(x) dS_{x} + \int_{\Sigma}(\widetilde{u}-\varphi)_{-}\widetilde{\phi}\mu(x) dS_{x} + C\theta +C\theta^{\frac{1}{2}}.
\end{align*}
Since $ \theta $ is arbitrarily small, we can further obtain
\begin{equation}\label{eq:a1}
\begin{split}
 \liminf_{\epsilon \rightarrow 0}\frac{1}{2}\int_{D}|{\gamma\nabla u^{\epsilon}}|^{2}dx \geq \ & \  \liminf_{\epsilon \rightarrow 0} \left(\frac{1}{2}\int_{\omega_{\epsilon} \leq \theta}|\gamma\nabla u^{\epsilon}_{-}|^{2} dx + \frac{1}{2}\int_{\omega_{\epsilon}> \theta}|\gamma \nabla u^{\epsilon}_{-}|^{2} dx \right) \\
\geq \ & \int_{D}(\gamma \nabla\phi)\cdot(\gamma \nabla\widetilde{u}_{-})dx - \frac{1}{2}\int_{D}|\gamma \nabla\phi|^{2} dx \\
- \ &  \frac{1}{2}\int_{\Sigma}\widetilde{\phi}^{2}\mu(x) dS_{x} + \int_{\Sigma}(\widetilde{u}-\varphi)_{-}\widetilde{\phi}\mu(x) dS_{x}
\end{split}
\end{equation}
If we choose $ \phi = \widetilde{u}_{-} $ and $ \widetilde{\phi} = (\widetilde{u} - \varphi )_{-}  $ in \eqref{eq:a1}, we can readily have
$$ \liminf_{\epsilon \rightarrow 0}\int_{D}|{\gamma\nabla u^{\epsilon}}|^{2}dx \ \geq \  \int_{D}|{\gamma\nabla \widetilde{u}}|^{2}dx+\int_{\Sigma}(\widetilde{u}-\varphi)^{2}_{-}\mu(x)dS_{x} . $$
Hence, it is sufficient for us to establish the estimate (\ref{na6}) and (\ref{na7}). Next, we shall first prove (\ref{na6}).

From Young's inequality, we have the following estimate:
$$ \int_{\omega_{\epsilon} < \theta } (\gamma \nabla\phi)\cdot(\gamma \nabla u^{\epsilon}_{-}) \leq  \frac{1}{2} \int_{\omega_{\epsilon} < \theta }|\gamma \nabla\phi|^{2} + \frac{1}{2} \int_{\omega_{\epsilon} < \theta }|\gamma \nabla u^{\epsilon}_{-} |^{2} .       $$
Since $ \omega_{\epsilon} \rightharpoonup 0 $ in $ H^{1}(D) $, so $ |{\omega_{\epsilon}} >\theta | \rightarrow 0   $ when $ \epsilon $ goes to zero.
Therefore
$$\lim_{\epsilon \rightarrow 0}\int_{\omega^{\epsilon}> \theta} |\gamma\nabla\phi|^{2} dx \ = \ 0 . $$
By H\"{o}lder's inequality, it holds that
$$ \int_{\omega_{\epsilon} > \theta} (\gamma \nabla\phi)\cdot (\gamma \nabla u^{\epsilon}_{-}) dx \leq \ \left(\int_{\omega_{\epsilon}> \theta}|\gamma\nabla\phi|^{2}\right)^{\frac{1}{2}} \left(\int_{\omega_{\epsilon}> \theta}|\gamma\nabla u^{\epsilon}_{-}|^{2}\right)^{\frac{1}{2}},$$
which implies that
$$  \lim_{\epsilon \rightarrow 0}\int_{\omega^{\epsilon}> \theta} (\gamma \nabla\phi)\cdot (\gamma \nabla u^{\epsilon}_{-}) dx \ = \ 0. $$
Since $ u^{\epsilon}_{-} \rightharpoonup \widetilde{u}_{-} $ in $ H^{1}(D) $ , we then arrive at the estimate (\ref{na6}).

We proceed to establish the estimate (\ref{na7}). From Young's inequality, we have that
$$ -\int_{\omega^{\epsilon} > \theta} (\gamma\nabla \omega_{\epsilon}) \cdot (\gamma \nabla u^{\epsilon}_{-}) \widetilde{\phi} dx \ \leq \ \frac{1}{2}\int_{\omega^{\epsilon} > \theta} |\gamma \nabla \omega_{\epsilon}\widetilde{\phi}|^{2}dx +  \frac{1}{2} \int_{\omega^{\epsilon} > \theta}|\gamma \nabla u^{\epsilon}_{-}|^{2}dx   $$
which in turn implies
\begin{equation}\label{eq:a2}
\begin{split}
& \frac{1}{2} \int_{\omega^{\epsilon} > \theta}|\gamma \nabla u^{\epsilon}_{-}|^{2}dx\\
  \geq& \   -(\int_{D}-\int_{\omega^{\epsilon} < \theta})(\gamma\nabla \omega_{\epsilon}) \cdot (\gamma \nabla u^{\epsilon}_{-}) \widetilde{\phi} dx
 -\frac{1}{2}(\int_{D}-\int_{\omega^{\epsilon}<\theta})|\gamma \nabla \omega_{\epsilon}\widetilde{\phi}|^{2}dx.
 \end{split}
 \end{equation}
Now we set
\begin{equation}\label{eq:A}
A:= \int_{\omega^{\epsilon} < \theta }(\gamma\nabla \omega_{\epsilon}) \cdot (\gamma \nabla u^{\epsilon}_{-}) \widetilde{\phi}\, dx,
\end{equation}
and
\begin{equation}\label{eq:B}
 B:= \int_{\omega^{\epsilon} < \theta}|\gamma \nabla \omega_{\epsilon}\widetilde{\phi}|^{2}dx .
 \end{equation}

For $ A $ in \eqref{eq:A}, we apply the H\"{o}lder's inequality to obtain
\begin{equation}\label{eq:a3}
\begin{split}
 A \ \leq \ &\left(\int_{\omega^{\epsilon} < \theta }|\gamma\nabla \omega_{\epsilon}\widetilde{\phi}|^{2}\right)^{\frac{1}{2}} \left(\int_{\omega^{\epsilon} < \theta }|\gamma \nabla u^{\epsilon}_{-}|^{2}\right)^{\frac{1}{2}}    \\
 \leq \ & C\left(\int_{\omega^{\epsilon} < \theta }|\gamma\nabla \omega_{\epsilon}\widetilde{\phi}|^{2}\right)^{\frac{1}{2}}
\end{split}
\end{equation}
where $ C $ is a constant depending only on $ b $ in the assumption (a1). Next, we prove that $B \ \leq \ C\theta $, which together with \eqref{eq:a3} readily yields \eqref{na7}.

Let $ \omega_{\theta}^{\epsilon} = (\theta-\omega_{\epsilon})_{+} $. Then $ \omega_{\theta}^{\epsilon} \in H_{0}^{1}(D) $ and $\omega_{\theta}^{\epsilon}\rightharpoonup \theta \  \text{in} \ H_{0}^{1}(D) $. By integration by parts,
$$ \lim_{\epsilon \rightarrow 0} \int_{D}(\gamma \nabla \omega_{\epsilon})\cdot (\gamma \nabla(\omega_{\theta}^{\epsilon}\widetilde{\phi}^{2}))dx\ = \ 0 ,$$
where we have used the fact that
$$ \nabla \cdot (\gamma^{T}\gamma\nabla \omega_{\epsilon})\ = \ 0.   $$
Thus
$$\lim_{\epsilon \rightarrow 0}\int_{D \cap \{\omega_{\epsilon}<\theta\}}|\gamma \nabla \omega_{\epsilon}|^{2}\widetilde{\phi}^{2} dx\ = \ \lim_{\epsilon \rightarrow 0}\int_{D \cap \{\omega_{\epsilon}<\theta\}}(\gamma \nabla \omega_{\epsilon})\cdot (\gamma \nabla \widetilde{\phi}^{2} )\omega_{\theta}^{\epsilon}dx. $$
Since $ \omega_{\epsilon} $ is bounded in $ H^{1}(D) $, we can apply the H\"{o}lder's inequality to obtain
$$ \int_{D \cap \{\omega_{\epsilon}<\theta\}}(\gamma \nabla \omega_{\epsilon})\cdot (\gamma \nabla \widetilde{\phi}^{2} )\omega_{\theta}^{\epsilon}dx\ \leq \ C\left(\int_{D}|\omega_{\theta}^{\epsilon}|^{2}\right)^{\frac{1}{2}} .$$
Noting $ \omega_{\epsilon} \rightharpoonup 0 $ in $ H^{1}(D) $, by the Sobolev embedding theorem \cite{LCE98}, we can further obtain that
$$ \lim_{\epsilon \rightarrow 0} \left(\int_{D}|\omega_{\theta}^{\epsilon}|^{2}dx\right)^{\frac{1}{2}} \ = \ \theta .$$
Therefore, it holds that
$$   B \ = \  \int_{\omega^{\epsilon} < \theta}|\gamma \nabla \omega_{\epsilon}\widetilde{\phi}|^{2}dx\ \leq \ C\theta   $$
which completes the proof of estimate (\ref{na7}).

The proof is complete.
\end{proof}

With Proposition~\ref{prop} and Lemmas~\ref{le2} and~\ref{le3}, one can handily prove Theorem~\ref{thm} by following a similar argument to that in {\cite{tang12}}, which we only sketch as follows.

\begin{proof}[Proof of Theorem~\ref{thm}]

The proof can be divided into three steps:

{\bf Step (i)}~ Choose an arbitrary function $ v \in H^{1}(D) $ such that $ v|_{\Gamma} \ = \ \psi(x) $ , and we can guarantee that $  v+(v-\varphi)_{-}\omega^{\epsilon} \in K $ under some assumption. Thus
$$ J(v+(v-\varphi)_{-}\omega_{\epsilon}) \  \geq  \ J(u^{\epsilon}). $$

{\bf Step (ii)}~By Proposition \ref{prop} and Lemma \ref{le2}, we can show
$$  \lim_{\epsilon \rightarrow 0} J(v+(v-\varphi)_{-}\omega_{\epsilon}) \ \leq J_{\mu}(v), $$
and hence $$  \limsup_{\epsilon \rightarrow 0} J(u_{\epsilon}) \ \leq J_{\mu}(v).  $$

{\bf Step (iii)}~By Lemma \ref{le3}, we have
$$ \liminf_{\epsilon \rightarrow 0} J(u_{\epsilon}) \ \geq J_{\mu}(\widetilde{u}) , $$
and hence
$$ J_{\mu}(\widetilde{u}) \ \leq \ J_{\mu}(v).   $$
\end{proof}


\section{construction of the corrector and its property}

Let us first introduce the following Laplace-Beltrami operator associated with a Riemannian metric $g(x)=(g_{ij}(x))$, which is a symmetric positive definite matrix-valued function in local coordinates:
\begin{equation}\label{eq:lb1}
\triangle_{g} u = |g|^{-1/2} \sum_{i,j=1}^{n} \frac{\partial}{\partial x^{i}}\left(|g|^{1/2}g^{ij}\frac{\partial u}{\partial x^{j}} \right),\ \ x=(x^j)_{j=1}^n\in\mathbb{R}^n,
\end{equation}
where $ G = \det(g)$ and $ (g^{ij}) = (g_{ij})^{-1}$. By {\cite{lh15,ODell}}, we know that there exists Green's function $ \Phi_{g}(x,y) $ satisfying
\begin{equation*}
    \left\{
    \begin{aligned}
        & -\triangle_{g}\Phi_{g}(x,y) \ = \ \delta (x-y),  \ \ x,y \in  \mathbb{R}^{n}, \\
        & \frac{\partial \Phi_{g}(x,y)}{\partial |x|} =O(|x|^{-2}) \ \mbox{as}\ |x|\rightarrow\infty, \\
    \end{aligned}
    \right.
\end{equation*}
and possessing the following expansion in a small open neighborhood of $ y $ :
\begin{align*}
\Phi_{g}(x,y) \sim  C_{1}(x){d_{g}(x,y)^{2-n}} + \ & \ C_{2}(x)d_{g}(x,y)^{3-n}+ \cdots + C_{n-2}(x)d_{g}(x,y)^{-1} \\
  \ + \ & \ C_{n}(x)d_{g}(x,y) \ + \ \cdots,
\end{align*}
and
$$ \nabla_{x} \Phi_{g}(x,y) \sim C_{1}^{\prime}(x){d_{g}(x,y)^{1-n}}V_{x,y} + C_{2}^{\prime}(x){d_{g}(x,y)^{2-n}}V_{x,y}+\cdots , $$
where $ d_{g}(x,y) $ is the Riemannian distance function and $ C_{j}(x) \in C^{\infty} $ for all $ j \in \mathbb N $. In what follows, we treat $\gamma$ equivalent to $|g|^{1/2} g^{-1}$.

Next, we begin to construct corrector $ \omega_{\epsilon} $ . Define
\begin{equation}\label{eq:osf1}
    \omega_{\epsilon,k}(x) : \ = \  \left\{
    \begin{aligned}
        &1,\ \  \ \ \ \ \ \ \ \ \ \ \ \ \ \ \ \ \ \ \ \ \ \ \ \ \ d_{g}(x,x_{\epsilon,k})\leq r_{\epsilon,k},\\
        & \frac{\Phi_{g}(x,x_{\epsilon,k})-\epsilon^{2-n}}{r_{\epsilon,k}^{2-n}-\epsilon^{2-n}}, \ \ \ \ \ r_{\epsilon,k}<d_{g}(x,x_{\epsilon,k})<\epsilon, \\
        &0. \ \ \ \ \ \ \ \ \ \ \ \ \ \  \ \ \ \ \ \ \ \ \ \ \ \ \ \ otherwise. \\
    \end{aligned}
    \right.
\end{equation}
Evidently, the support of \ $ \omega_{\epsilon,k} $  is situated in the ball \ $ B_{\epsilon}(x_{\epsilon,k}) $, and these balls are disjoint through our assumptions. Thus we can select our corrector as follows :
$$ \omega_{\epsilon}: =\sum_{k}\omega_{\epsilon,k}. $$

It is easy to see that $ ||\omega_{\epsilon}||_{L^{\infty}(D)}\ \leq \ C  $. We proceed to verify the conditions (\ref{pr4}) and (\ref{pr5}) in Proposition~\ref{prop}. Firstly, we can obtain the following result :
\begin{eqnarray}
\label{31}
\lim_{\epsilon\rightarrow 0}\int_{D}|\omega_{\epsilon}|^{2}dx \ = \ 0 .
\end{eqnarray}
In fact, one can deduce that
\begin{align*}
& \int_{B_{\epsilon}}|\omega_{\epsilon}|^{2}dx \ = \ |B_{r_{\epsilon}}|+\int_{B_{\epsilon}\setminus B_{r_{\epsilon}}}|\omega_{\epsilon}|^{2}dx \\
 \leq \ &|B_{r_{\epsilon}}|+\frac{C}{(r_{\epsilon}^{2-n}-\epsilon^{2-n})^{2}}\int_{B_{\epsilon}\backslash B_{r_{\epsilon}}}\left(d_{g}(x,0)^{2-n}-\epsilon^{2-n}+\sum_{j=3,j\neq n}^{\infty}d_{g}(x,0)^{j-n}\right)^{2}dx   \\
\leq \ &|B_{r_{\epsilon}}|+\frac{C(n)}{(r_{\epsilon}^{2-n}-\epsilon^{2-n})^{2}} \int_{r_{\epsilon}}^{\epsilon}\bigg((r^{2-n}-\epsilon^{2-n})^{2}r^{n-1}+ \ \sum_{j=3,j\neq n}^{\infty}r^{2j-2n}r^{n-1}  \\
& \ \ \ \ \ \ \ \ \ \ \ \ \ \ \ \ \ \ \ \ \ \ \ \ \ \ \ \ \ \ \ \ \ \ \  + (r^{2-n}-\epsilon^{2-n})\sum_{j=3,j\neq n}^{\infty}r^{j-n}r^{n-1}\bigg) dr   \\
\leq \ &|B_{r_{\epsilon}}|+C(n) O(\epsilon^{n})+ \frac{C(n)}{(r_{\epsilon}^{2-n}-\epsilon^{2-n})^{2}}\int_{r_{\epsilon}}^{\epsilon}\left(\sum_{j=3}^{\infty}r^{2j-n-1} +(r^{2-n}-\epsilon^{2-n})\sum_{j=3}^{\infty}r^{j-1}\right)dr     \\
\leq \ & |B_{r_{\epsilon}}|+ C(n)\left( O(\epsilon^{n})+\sum_{j=4}^{\infty}\max\bigg\{O(\epsilon^{n+j}),O(\epsilon^{\frac{(n-1)(n+j-2)}{n-2}})\bigg\}+\sum_{j=1}^{\infty}O(\epsilon^{2n+j})\right) \\
\leq \ & C(n)\left( O(\epsilon^{n})+\max\bigg\{ O(\epsilon^{n+3}),O(\epsilon^{\frac{n(n-1)}{n-2}})\bigg\}+O(\epsilon^{2n}) \right) \\
\leq \ & C(n)O(\epsilon^{n}),
\end{align*}
which readily yields (\ref{31}).

Secondly, we can estimate the term $ ||\nabla \omega_{\epsilon}||_{L^{2}(D)} $ as follows:
\begin{align*}
& \int_{B_{\epsilon}}|\nabla\omega_{\epsilon}|^{2}dx \ \leq \ \frac{C}{(r_{\epsilon}^{2-n}-\epsilon^{2-n})^{2}}\int_{B_{\epsilon}\setminus B_{r_{\epsilon}}}\sum_{j=1,j\neq n}^{\infty}d_{g}(x,0)^{2j-2n} dx   \\
\leq \ &\frac{C(n)}{(r_{\epsilon}^{2-n}-\epsilon^{2-n})^{2}} \int_{r_{\epsilon}}^{\epsilon}\left(r^{1-n} +\sum_{j=2}^{\infty}r^{2j-n-1}\right)dr  \\
\leq \ & C(n)\left(O(\epsilon^{n-1})+\sum_{j=2}^{\infty}\max\bigg\{O(\epsilon^{n+j}),O(\epsilon^{\frac{(n+j-2)(n-1)}{n-2}})\bigg\}\right)  \\
\leq \ & C(n)\left(O(\epsilon^{n-1})+\max\bigg\{O(\epsilon^{n+1}),O(\epsilon^{\frac{(n-1)^{2}}{n-2}})\bigg\}\right)  \\
\leq \ & C(n)O(\epsilon^{n-1})
\end{align*}
which in turn implies that
\begin{eqnarray}
\label{32}
\int_{D}|\nabla \omega_{\epsilon}|^{2}dx \ \leq \ C .
\end{eqnarray}
Finally, we can combine the estimates (\ref{31}) and (\ref{32}) to derive that
$$ \omega_{\epsilon}\rightharpoonup 0 \ \ \text{in} \ \  H^{1}(D), $$
which readily verifies the condition (\ref{pr4}).

\section{Limiting property of the corrector}

To complete the proof of  Proposition \ref{prop}, it suffices to verify the inequality (\ref{pr5}). For simplicity, we write  $$ B_{\epsilon, k}\ :=\ B_{\epsilon}(x_{\epsilon,k})\setminus B_{r_{\epsilon,k}}(x_{\epsilon,k}), \ \ D_{\epsilon, k}\ :=\ (B_{\epsilon}(x_{\epsilon,k})\setminus B_{r_{\epsilon,k}}(x_{\epsilon,k})) \cap D.$$
If $ \epsilon $ is sufficiently small, we clearly have
$$ \int_{D}(\gamma\nabla\omega_{\epsilon})\cdot(\gamma\nabla v_{\epsilon})\phi \ dx=\sum_{k}\int_{D_{\epsilon,k}}(\gamma\nabla\omega_{\epsilon})\cdot(\gamma\nabla v_{\epsilon}) \phi dx . $$

Integrating by parts, it holds that
\begin{align}
\label{41}
&\int_{D_{\epsilon, k}}(\gamma\nabla\omega_{\epsilon})\cdot(\gamma\nabla v_{\epsilon})\phi dx    \nonumber   \\
= \ &\int_{D_{\epsilon, k}} \nabla\cdot(\gamma^{T}\gamma\nabla \omega_{\epsilon}v_{\epsilon}\phi )dx-\int_{D_{\epsilon, k}}(\gamma\nabla\omega_{\epsilon})\cdot(\gamma\nabla\phi )v_{\epsilon}dx   \nonumber         \\
= \ &\int_{\partial D_{\epsilon, k}}(\gamma^{T}\gamma\nabla\omega_{\epsilon})\cdot \overrightarrow{\nu}v_{\epsilon}\phi dS_{x}-\int_{D_{\epsilon, k}}(\gamma\nabla\omega_{\epsilon})\cdot(\gamma\nabla\phi )v_{\epsilon}dx  \nonumber        \\
= \ &\int_{\partial D_{\epsilon, k}}(\gamma\nabla\omega_{\epsilon})\cdot (\gamma\overrightarrow{\nu})v_{\epsilon}\phi dS_{x}-\int_{D_{\epsilon, k}}(\gamma\nabla\omega_{\epsilon})\cdot(\gamma\nabla\phi )v_{\epsilon}dx,
\end{align}
where we make use of the following fact :
\begin{align*}
 \nabla\cdot(\gamma^{T}\gamma\nabla \omega_{\epsilon})=0,\ \ \forall x\in  D_{\epsilon, k}.
\end{align*}
In addition, we can get the following result :
\begin{align}
\label{42}
 \lim_{\epsilon \rightarrow 0}\sum_{k}\int_{D_{\epsilon, k}}(\gamma\nabla\omega_{\epsilon})\cdot(\gamma\nabla\phi) v_{\epsilon}dx \ = \ 0.
\end{align}
In fact, by direct calculations, we have
\begin{align*}
 \int_{B_{\epsilon}}|\gamma\nabla\omega_{\epsilon}|dx  \leq \ \ & \frac{C}{r_{\epsilon}^{2-n}-\epsilon^{2-n}}\int_{B_{\epsilon}\backslash B_{r_{\epsilon}}}\sum_{j=1,j\neq n}^{\infty}d_{g}(x,0)^{j-n} dx      \\
\leq \ \ &\frac{C(n)}{r_{\epsilon}^{2-n}-\epsilon^{2-n}}\int_{r_{\epsilon}}^{\epsilon} \sum_{j=1}^{\infty}r^{j-1}dr   \\
\leq \ \ &C(n)O(\epsilon^{n})+C(n)\sum_{j=1}^{\infty}O(\epsilon^{n+j})     \\
\leq \ \ &C(n)O(\epsilon^{n}),
\end{align*}
and hence
\begin{align*}
  \int_{D}|\gamma \nabla\omega_{\epsilon}||\gamma\nabla \phi|v_{\epsilon} dx \ \leq \ C\int_{D}|\gamma\nabla \omega_{\epsilon}|dx \ \leq \ C\epsilon
\end{align*}
where $ C $ depends only on $ n $, $ a $ and $ b $. Consequently, we have established (\ref{42}). This, together with (\ref{41}), inspires us to only consider the following integral
\begin{align}
\label{eq1}
&\int_{\partial D_{\epsilon, k}}(\gamma\nabla\omega_{\epsilon})\cdot (\gamma\overrightarrow{\nu})v_{\epsilon}\phi \ dS_{x} \nonumber \\
 =&\left(\int_{\partial B_{r_\epsilon}(x_{\epsilon, k})\cap D}+\int_{\partial D \cap B_{\epsilon, k}} +\int_{\partial B_{\epsilon}(x_{\epsilon,k})\cap D}\right)(\gamma\nabla\omega_{\epsilon})\cdot (\gamma\overrightarrow{\nu})v_{\epsilon}\phi \ dS_{x}  \nonumber \\
 =&\ : E_{1} + E_{2} + E_{3} .
\end{align}
Noticing that
\begin{align}
\label{44}
 \int_{\partial B_{r_\epsilon}\cap D}(\gamma\nabla\omega_{\epsilon})\cdot (\gamma\overrightarrow{\nu})v_{\epsilon}\phi dS_{x} \ \geq \ & \int_{\partial B_{r_{\epsilon}}\cap D}\frac{\sum_{i=1}C'_{i}r_{\epsilon}^{i-n}}{r_{\epsilon}^{2}(r_{\epsilon}^{2-n}-\epsilon^{2-n})}(\gamma x)^{T}(\gamma x)v_{\epsilon}\phi dS_{x}     \nonumber \\
 \geq \  & \frac{-C||v_{\epsilon}\phi||_{L^{\infty}}}{r_{\epsilon}^{2-n}-\epsilon^{2-n}}\int_{\partial B_{r_{\epsilon}}\cap D}\sum_{i=1}^{+\infty}r_{\epsilon}^{i-n}dS_{x}       \nonumber \\
  \geq \  & -\frac{C\epsilon^{n-1}}{1-r_{\epsilon}}.
\end{align}
where $ \{C'_{i}\}_{i=1}^{\infty} $ are bounded. Hence, we see that
$$ \lim_{\epsilon \rightarrow 0}\sum_{k}\int_{\partial B_{r_\epsilon}(x_{\epsilon, k})\cap D} (\gamma\nabla\omega_{\epsilon})\cdot (\gamma\overrightarrow{\nu})v_{\epsilon}\phi \ dS_{x} \ \geq \ 0 .  $$
In the following, we need to estimate $ E_{2} $ and $ E_{3} $. For the estimation of $ E_{2} $, the key idea is to make a local flattening of the boundary $ \partial D $ because of $ \partial D \in C^{1,\alpha} $. More precisely, we have the following lemma.
\vspace{5pt}
\begin{lemma}
\label{le41}
 Let $ \omega_{\epsilon}$ be defined in \eqref{eq:osf1}. Then
\begin{eqnarray*}&& \lim_{\epsilon\rightarrow0}\sum_{k}\int_{\partial D \cap B_{\epsilon, k}} \big(\gamma\nabla\omega_{\epsilon}\big)\cdot\big(\gamma\overrightarrow{\nu}\big)v_{\epsilon}\phi dS_{x} \ \ \ =\ \  0 \ . \\
\end{eqnarray*}
\end{lemma}
\begin{proof} \ \ Without loss of generality, we may suppose that $ x_{\epsilon, k}$  is the origin. Since $ \partial D \in C^{1,\alpha}$,  then  near the origin, $ \partial D $  can be given by $$ x_{n} \ \ = \ \ \rho(x') ,$$
where $ \rho $ is $ C^{1,\alpha}$ near the origin in $\mathbb{R}^{n-1}$  and $|\rho(x')|\leq C|x'|^{1+\alpha} $  for $|x'|$ small enough. Furthermore,  there holds
  $$D\cap {B}_{\epsilon}\subset \{x_n \ >\ \rho(x')\} \ \text{for} \ \epsilon \ \text{small}.$$
Let $\nu(x')$ be the outward unit normal vector on $\partial D$ at $ (x',\rho (x'))$ near the origin.Then $ \nu \in C^{\alpha} $ near the origin in $\mathbb{R}^{n-1}$. Substituting by variables, we have
\begin{eqnarray}
\label{45}
&&\int_{\partial D \cap (B_{\epsilon}\setminus B_{r_{\epsilon}})}\big(\gamma\nabla\omega_{\epsilon}\big)\cdot\big(\gamma\overrightarrow{\nu}\big)v_{\epsilon}\phi dS_{x} \nonumber \\
&&=\int_{(B_{\epsilon} \setminus B_{r_{\epsilon}} )}\bigg(\gamma\nabla\omega_{\epsilon}(x',\rho(x'))\bigg)\cdot\bigg(\gamma \overrightarrow{\nu}(x')\bigg)v_{\epsilon}\phi\sqrt{1+|\nabla\rho(x')|^{2}}dx' .
\end{eqnarray}
Next, we split the term $\bigg(\gamma\nabla\omega_{\epsilon}(x',\rho(x'))\bigg)\cdot\bigg(\gamma \overrightarrow{\nu}(x')\bigg) $ into three parts :
\begin{align}
\label{46}
\bigg(\gamma\nabla\omega_{\epsilon}(x',\rho(x'))\bigg)\cdot\bigg(\gamma \overrightarrow{\nu}(x')\bigg) \ & = \  \bigg(\gamma\nabla \omega_{\epsilon}(x',0)\bigg)\cdot \bigg(\gamma\overrightarrow{\nu}(0)\bigg) \nonumber \\
& + \ \bigg(\gamma\nabla \omega_{\epsilon}(x',\rho(x'))-\gamma\nabla \omega_{\epsilon}(x',0)\bigg) \cdot \bigg(\gamma\overrightarrow{\nu}(0)\bigg) \nonumber \\
& +  \ \bigg(\gamma\nabla\omega_{\epsilon}(x',\rho(x'))\bigg)\cdot\bigg(\gamma\overrightarrow{\nu}(x')-\gamma\overrightarrow{\nu}(0)\bigg) \nonumber   \\
& = \  : \ A_{1} \ + \ A_{2} \ + \ A_{3}
\end{align}
We note that $ A_{1} = \ 0 $ by the construction of $ \omega_{\epsilon} $. Hence, we only need to focus on the estimates of $ A_{2} $ and $ A_{3}$.

To facilitate the estimate of the term $ A_{2} $ , we define
\begin{eqnarray}
\label{47}
 A_{2,1}\ :=\ \sum_{i=1}^{n-1}|\frac{\partial \omega_{\epsilon}}{\partial x_{i}}(x',\rho(x'))-\frac{\partial \omega_{\epsilon}}{\partial x_{i}}(x',0)|,\ \ A_{2,2}\ :=\ |\frac{\partial \omega_{\epsilon}}{\partial x_{n}}(x',\rho(x'))-\frac{\partial \omega_{\epsilon}}{\partial x_{n}}(x',0)|.
\end{eqnarray}
Then we have by direct computations that
\begin{align}
\label{48}
  &A_{2,1}: =\sum_{i=1}^{n-1}\left|\frac{\partial \omega_{\epsilon}}{\partial x_{i}}(x',\rho(x'))-\frac{\partial \omega_{\epsilon}}{\partial x_{i}}(x',0)\right| \nonumber \\
\leq & \sum_{i=1}^{n-1}\sup_{0\leq x_{n}\leq \rho(x')}\left|\frac{\partial^{2}\omega_{\epsilon}}{\partial x_{i}\partial x_{n}}(x',x_{n})\rho(x')\right|  \nonumber \\
\leq & \frac{C}{r_{\epsilon}^{2-n}-\epsilon^{2-n}}\sum_{i=1}^{n-1}\sup_{0\leq x_{n}\leq \rho(x')} \sum_{j=1}^{\infty}\bigg\{d_{g}(x,0)^{j+1-n} +(|x_{n}|+|x_{i}|)d_{g}(x,0)^{j-1-n} \nonumber \\
 &\ \ \ \ \ \ \ \ \ \ \ \ \ \ \ \ \ \ \ \ \ \ \ \ \ \ \ \ + |x_{n}||x_{i}|d_{g}(x,0)^{j-3-n}\bigg\} |\rho(x')|     \nonumber \\
 \leq & \frac{C}{r_{\epsilon}^{2-n}-\epsilon^{2-n}}\sum_{i=1}^{n-1}\sup_{0\leq x_{n}\leq \rho(x')} \bigg\{ |x'|^{1-n}+ M +(|x_{n}|+|x_{i}|)(|x'|^{-n-1}+M) \nonumber \\
 & \ \ \ \ \ \ \ \ \ \ \ \ \ \ \ \ \ \ \ \ \ \ \ \ \ \ \ \ +|x_{n}||x_{i}|(|x'|^{-3-n}+M)\bigg\}|\rho(x')| \ \  \left( M=\frac{1}{1-(|x'|+|x'|^{1+\alpha})} \right)   \nonumber \\
 \leq & \frac{C(n)}{r_{\epsilon}^{2-n}-\epsilon^{2-n}}\bigg\{|x'|^{1-n}+M+(|x'|^{1+\alpha}+|x'|)(|x'|^{-n-1}+M)   \nonumber \\ & \ \ \ \ \ \ \ \ \ \ \ \ \ \ \ \ \ \ \ \ \ \ \ \ \ \ \ \ \
 + |x'|^{2+\alpha}(|x'|^{-3-n}+M)\bigg\}|x'|^{1+\alpha}    \nonumber \\
  \leq & \frac{C(n)}{r_{\epsilon}^{2-n}-\epsilon^{2-n}}\bigg\{|x'|^{\alpha}(1+|x'|+|x'|^{1+\alpha}+|x'|^{\alpha+2})+|x'|^{\alpha-n+2} + |x'|^{\alpha-n+1} \nonumber \\ & \ \ \ \ \ \ \ \ \ \ \ \ \ \ \ \ \ \ \ \ \ \ \ \ \ \ \ \ \ \  +|x'|^{1+2\alpha-n}+|x'|^{2\alpha-n}\bigg\}.
\end{align}
Similarly, we can show that
\begin{align}
\label{49}
  &A_{2,2}:=\left|\frac{\partial \omega_{\epsilon}}{\partial x_{n}}(x',\rho(x'))-\frac{\partial \omega_{\epsilon}}{\partial x_{n}}(x',0)\right| \nonumber \\
\leq & \frac{C}{r_{\epsilon}^{2-n}-\epsilon^{2-n}}\bigg\{|x'|^{\alpha}(1+|x'|^{1+\alpha}+|x'|^{2\alpha+2})+|x'|^{\alpha-n+2} + |x'|^{1+2\alpha-n} \nonumber \\ & \ \ \ \ \ \ \ \ \ \ \ \ \ \ \ \ \ \ \ \ \ \ \ \ \ \ \ \ +|x'|^{3\alpha-n}\bigg\} .
\end{align}
Then we have by combining the two estimates (\ref{48}) and (\ref{49}) that
\begin{align}
\label{410}
  & \left|A_{2}\right| \ \leq \ \frac{C}{r_{\epsilon}^{2-n}-r^{2-n}} \bigg\{|x'|^{\alpha}(1+|x'|+|x'|^{1+\alpha}+|x'|^{\alpha+2}+|x'|^{2+2\alpha})+|x'|^{\alpha-n+1}+|x'|^{\alpha-n+2} \nonumber \\ & \ \ \ \ \ \ \ \ \ \ \ \ \ \ \ \ \ \ \ \ \ \ \ \ \ \ \ \ \ \ \ \ \ \ \ \ \ \
+ |x'|^{2\alpha-n}+|x'|^{2\alpha-n+1}+|x'|^{3\alpha-n}\bigg\}.
\end{align}
where $ C $ depends only on $ n $, $ a $, $ b $, and $ |C'_{i}|_{L^{\infty}(D)}(i=1,2\cdots) $.

For the term $ A_{3} $, let us first recall that $ \partial D \in C^{1,\alpha}$, then we will select $ \epsilon $ sufficiently small such that $ |\nabla \rho |\leq 1 $ for $ |x'|< \epsilon $, and
$$ \bigg|\gamma\bigg(\overrightarrow{\nu}(x')-\overrightarrow{\nu}(0)\bigg)\bigg| \ \leq \ C|x'|^{\alpha}.$$
Hence
\begin{align}
\label{411}
|A_{3}| \ \leq \ C |x'|^{\alpha}\bigg|\gamma\nabla\omega_{\epsilon}(x',\rho(x'))\bigg|.
\end{align}
In view of (\ref{46}), (\ref{410}) and (\ref{411}), it concludes that
\begin{align}
\label{412}
&\int_{(B_{\epsilon} \setminus B_{r_{\epsilon}} )}\bigg(\gamma\nabla\omega_{\epsilon}(x',\rho(x'))\bigg)\cdot\bigg(\gamma \overrightarrow{\nu}(x')\bigg)v_{\epsilon}\phi\sqrt{1+|\nabla\rho(x')|^{2}}dx'  \nonumber \\
\leq \ \ &  2||\phi v_{\epsilon}||_{L^{\infty}}\int_{(B_{\epsilon}\setminus B_{r_{\epsilon}})}\bigg|\bigg(\gamma\nabla\omega_{\epsilon}(x',\rho(x'))\bigg)\cdot \bigg(\gamma\overrightarrow{\nu}(x')\bigg)\bigg|dx' \nonumber \\
\leq \ \ & 2||\phi v_{\epsilon}||_{L^{\infty}}\int_{(B_{\epsilon}\setminus B_{r_{\epsilon}})}\bigg|\gamma\nabla\omega_{\epsilon}(x',\rho(x'))-\gamma\nabla \omega_{\epsilon}(x',0)\bigg|+C\bigg|\gamma\nabla \omega_{\epsilon}(x',\rho(x'))\bigg||x'|^{\alpha}dx'  \nonumber \\
\leq \ \ & 2||\phi v_{\epsilon}||_{L^{\infty}}\int_{(B_{\epsilon}\setminus B_{r_{\epsilon}})}\bigg\{|x'|^{\alpha}(1+|x'|+|x'|^{1+\alpha}+|x'|^{\alpha+2}+|x'|^{2+2\alpha})+|x'|^{\alpha-n+1}+|x'|^{\alpha-n+2} \nonumber \\ & \ \ \ \ \ \ \ \ \ \ \ \ \ \ \ \ \ \ \ \ \ \ \ \ \ \ \ \ \ \ \ \ \ \ \ \ \ \
+ |x'|^{2\alpha-n}+|x'|^{2\alpha-n+1}+|x'|^{3\alpha-n}\bigg\}dx' \nonumber \\
\leq \ \ &||\phi v_{\epsilon}||_{L^{\infty}}\bigg\{O(\epsilon^{\alpha+2n-2})(1+\epsilon)+O(\epsilon^{2\alpha+2n-1})(1+\epsilon+\epsilon^{\alpha+1})+O(\epsilon^{\alpha+n-1})(1+\epsilon) \nonumber \\ & \ \ \ \ \ \ \ \ \ \ \ \ \ \ \ \ \ \ \ \ \ \ \ \ \ \ \ \ \ \ \ \ \ \ \  \ \ \ +O(\epsilon^{2\alpha+n-2})(1+\epsilon+\epsilon^{\alpha})\bigg\} \ \nonumber \\
\leq  \ \ & ||\phi v_{\epsilon}||_{L^{\infty}}O(\epsilon^{2\alpha+n-2}),
\end{align}
which, together with (\ref{45}) and $ \alpha \in (2^{-1},1)$, yields that
$$ \lim_{\epsilon\rightarrow0}\sum_{k}\int_{\partial D \cap B_{\epsilon, k}} (\gamma\nabla\omega_{\epsilon})\cdot(\gamma\overrightarrow{\nu})v_{\epsilon}\phi dS_{x} \ \ \ =\ \  0 \  .$$
Hence we complete the proof of Lemma \ref{le41}.
\end{proof}
 The next lemma tells us the estimate of $E_{3}$. It is worth noting that the construction of the auxiliary function $ q_{\epsilon}(x) $ plays an important role in its estimation.
\begin{lemma}
\label{le42}
 Suppose that $ \omega_{\epsilon}$ satisfies the definition of the above section. Then there holds
\begin{eqnarray*}&&  \lim_{\epsilon\rightarrow 0}\sum_{k}\int_{\partial B_{\epsilon}(x_{\epsilon,k})\cap D}(\gamma\nabla\omega_{\epsilon})\cdot(\gamma\overrightarrow{\nu})v_{\epsilon}\phi dS_{x}\ \ \ \geq \ \ \ -\int_{\Sigma}\phi v \mu(x)dS_{x},
\end{eqnarray*}
with equality held if $ v_{\epsilon}\big|_{T_{\epsilon}} \ = \ 0$.
\end{lemma}

\begin{proof} \ Suppose that $ \epsilon $ is sufficiently small. By direct calculations, we have
$$ \big(\gamma \nabla \omega_{\epsilon}\big)\cdot\big(\gamma \overrightarrow{\nu}\big)= \frac{\sum_{i=1}^{\infty}\widetilde{C_{i}}\epsilon^{i-n}}{\epsilon^{2}(r_{\epsilon}^{2-n}-\epsilon^{2-n})}(\gamma x)^{T}(\gamma x) = \frac{\sum_{i=1}^{\infty}\widetilde{C_{i}}\epsilon^{i-1}}{\epsilon^{2}(\widetilde{r_{\epsilon}}^{2-n}-\epsilon)}(\gamma x)^{T}(\gamma x) $$
on $ \partial B_{\epsilon} $, where $ r_{\epsilon}=\widetilde{r_{\epsilon}}\epsilon^{\frac{n-1}{n-2}}$ and $ \{ \widetilde{C_{i}}\}_{i=1}^{\infty} $ are bounded.

Motivated by Cioranescu and Murant {\cite{cm82a,cm82b}}, we consider the following  auxiliary function
\begin{equation*}
     q_{\epsilon}(x):= \left\{
     \begin{aligned}
     & \frac{\sum_{i=1}^{\infty}\widetilde{C_{i}}\epsilon^{i-1}}{2\epsilon(\widetilde{r_{\epsilon}}^{2-n}-\epsilon)}\bigg(d^{2}_{g}(x,0)-\epsilon^{2}\bigg),\ \ \ d_{g}(x,0)\leq\epsilon \\ \\
     & 0,\ \ \ \ \ \ \ \ \ \ \ \ \ \ \ \ \ \ \ \ \ \ \ \ \ \ \ \ \ \ \ \ \ \ \ \ \ \ \  d_{g}(x,0)>\epsilon,
     \end{aligned}
     \right.
\end{equation*}
which gives that
\begin{align}
\label{413}
(\gamma \nabla q_{\epsilon})\cdot (\gamma \overrightarrow{\nu})\ = \ (\gamma \nabla \omega_{\epsilon})\cdot (\gamma \overrightarrow{\nu}) \ \ \ \text{on} \ \ \partial B_{\epsilon},
\end{align}
and
\begin{align}
\label{414}
\lim_{\epsilon\rightarrow 0}\int_{D}|\nabla q_{\epsilon}(x)|^{2}dx \ = \ 0.
\end{align}
In fact, (\ref{413}) is obvious. In the following, we need to show (\ref{414}). In fact, direct calculation yields that
\begin{align*}
 \int_{B_{\epsilon}}|\nabla q_{\epsilon}(x)|^{2}dx \ \ = \ \ & \int_{B_{\epsilon}}\left(\frac{\sum_{i=1}^{\infty}\widetilde{C_{i}}\epsilon^{i-1}}{\widetilde{r_{\epsilon}}^{2-n}-\epsilon}\right)^{2}\frac{x^{2}}{\epsilon^{2}}dx \\
\leq \ \ &C(n)\left(\frac{\sum_{i=1}^{\infty}\widetilde{C_{i}}\epsilon^{i-1}}{(\widetilde{r_{\epsilon}}^{2-n}-\epsilon)\epsilon}\right)^{2}\int_{0}^{\epsilon}r^{n+1}dr  \\
\leq \ \ &C(n) \frac{(\sum_{i=1}^{\infty}\widetilde{C_{i}}\epsilon^{i-1})^{2}}{(\widetilde{r_{\epsilon}}^{2-n}-\epsilon)^{2}}\epsilon^{n}   \\
\leq \ \ &C(n)\left(\frac{1}{(1-\epsilon)(\widetilde{r_{\epsilon}}^{2-n}-\epsilon)}\right)^{2}\epsilon^{n}  \ \  ( c_{1} \leq \widetilde{r_{\epsilon}} \leq c_{2} )\\
\leq \ \ &C(n)O(\epsilon^{n})
\end{align*}
which verifies (\ref{414}).

Integrating by parts, we see that
\begin{align}
\label{415}
& \int_{D}\phi v_{\epsilon}\nabla\cdot(\gamma^{T}\gamma\nabla q_{\epsilon}) dx   \nonumber \\
 =& -\int_{D}\bigg(\gamma\nabla(\phi v_{\epsilon})\bigg)\cdot\bigg(\gamma\nabla q_{\epsilon}(x)\bigg) dx
+\sum_{k}\int_{\partial(B_{\epsilon}(x_{\epsilon,k})\cap D)}\big(\gamma \nabla q_{\epsilon}\big)\cdot\big(\gamma \overrightarrow{\nu}\big)\phi v_{\epsilon} dS_{x} .
\end{align}
Then by applying H\"{o}lder's inequality to the first term of (\ref{415}), we have
$$ \int_{D}\bigg(\gamma\nabla(\phi v_{\epsilon})\bigg)\cdot\bigg(\gamma\nabla q_{\epsilon}\bigg) dx \ \leq \ C(n)\big|\big|\gamma \nabla(\phi v_{\epsilon})\big|\big|_{L^{\infty}(D)}\int_{D}|\gamma \nabla q_{\epsilon}|^{2} dx,$$
which means that (by (\ref{414}))
$$ \lim_{\epsilon \rightarrow 0} \int_{D}\bigg(\gamma\nabla(\phi v_{\epsilon})\bigg)\cdot\bigg(\gamma\nabla q_{\epsilon}(x)\bigg) dx \ = \ 0 .$$
Therefore it suffices for us to only study the second term of (\ref{415}),
\begin{align}
\label{416}
& \sum_{k}\int_{\partial(B_{\epsilon}(x_{\epsilon,k})\cap D)}\big(\gamma \nabla q_{\epsilon}\big)\cdot\big(\gamma \overrightarrow{\nu}\big)\phi v_{\epsilon} dS_{x} \nonumber \\
= \ &\sum_{k}\left(\int_{\partial B_{\epsilon}(x_{\epsilon ,k})\cap D}+\int_{\partial D \cap B_{\epsilon}(x_{\epsilon,k})}\right)\big(\gamma \nabla q_{\epsilon})\cdot(\gamma \overrightarrow{\nu}\big)\phi v_{\epsilon} dS_{x} \nonumber \\
=& : \ E_{3,1} \ + \ E_{3,2}
\end{align}
For the term $ E_{3,2}$, we can use a similar argument to that for Lemma \ref{le41} to obtain
\begin{align}
\label{417}
\int_{\partial D \cap B_{\epsilon}}(\gamma \nabla q_{\epsilon})\cdot(\gamma \overrightarrow{\nu})\phi v_{\epsilon} dS_{x}
\leq \ & C||\phi v_{\epsilon}||_{L^{\infty}}\int_{\partial D \cap B_{\epsilon}}\frac{\sum_{i=1}^{\infty}\widetilde{C_{i}}\epsilon^{i-1}}{(\widetilde{r_{\epsilon}}^{2-n}-\epsilon)\epsilon} |x \cdot \overrightarrow{\nu}|dS_{x} \nonumber \\
\leq \ & C||\phi v_{\epsilon}||_{L^{\infty}} \int_{\partial D \cap B_{\epsilon}}\frac{\sum_{i=1}^{\infty}\widetilde{C_{i}}\epsilon^{i-1}}{(\widetilde{r_{\epsilon}}^{2-n}-\epsilon)\epsilon}|x'|^{1+\alpha}dS_{x}  \nonumber  \\
\leq \ & C||\phi v_{\epsilon}||_{L^{\infty}}\int_{B_{\epsilon}}\frac{\sum_{i=1}^{\infty}\widetilde{C_{i}}\epsilon^{i-1}}{(\widetilde{r_{\epsilon}}^{2-n}-\epsilon)\epsilon}|x'|^{1+\alpha}dx' \nonumber \\
\leq \ & C||\phi v_{\epsilon}||_{L^{\infty}} \frac{\sum_{i=1}^{\infty}\widetilde{C_{i}}\epsilon^{i-1}}{(\widetilde{r_{\epsilon}}^{2-n}-\epsilon)\epsilon}\int_{0}^{\epsilon}r^{n+\alpha-1}dr  \nonumber   \\
\leq \ & C||\phi v_{\epsilon}||_{L^{\infty}} \frac{\sum_{i=1}^{\infty}\widetilde{C_{i}}\epsilon^{i-1}}{\widetilde{r_{\epsilon}}^{2-n}-\epsilon }\epsilon^{n+\alpha-1}  \nonumber \\
\leq \ & \frac{C}{(\widetilde{r_{\epsilon}}^{2-n}-\epsilon)(1-\epsilon)}\epsilon^{n+\alpha-1} \ \ ( c_{1} \leq \widetilde{r_{\epsilon}} \leq c_{2} )   \nonumber \\
= \ & CO(\epsilon^{n+\alpha-1}),
\end{align}
where $ C $ is a positive constant depending only on $ n $, $ a,b $, $ ||\phi v_{\epsilon}||_{L^{\infty}(D)}$ and $ ||C'_{i}||_{L^{\infty}(D)}$. \\
Thus
 $$ \lim_{\epsilon \rightarrow 0}\sum_{k}\int_{\partial D \cap B_{\epsilon}(x_{\epsilon,k})}(\gamma \nabla q_{\epsilon})\cdot(\gamma \overrightarrow{\nu})\phi v_{\epsilon} dS_{x} \ = \ 0 .   $$
For the term $ E_{3,1}$, returning now to (\ref{415}), we need to compute
\begin{align}
\label{418}
\int_{D\cap B_{\epsilon}}\phi v_{\epsilon}\nabla\cdot(\gamma^{T}\gamma\nabla q_{\epsilon}) dx \ = \ & \int_{D \cap B_{\epsilon}} \frac{\sum_{i=1}^{\infty}\widetilde{C_{i}}\epsilon^{i-1}}{(\widetilde{r_{\epsilon}}^{2-n}-\epsilon)\epsilon} \nabla\cdot(\gamma^{T}\gamma x)\phi v_{\epsilon}dx   \nonumber  \\
= \ & -\int_{D}\frac{(\sum_{i=1}^{\infty}\widetilde{C_{i}}\epsilon^{i-1})\widetilde{r_{\epsilon}}^{n-2}}{\epsilon \widetilde{r_{\epsilon}}^{n-2}-1}\nabla\cdot(\gamma^{T}\gamma x)\frac{1}{\epsilon}\chi_{B_{\epsilon}}(x)\phi v_{\epsilon} dx.   \nonumber  \\
\end{align}
By (\ref{413}), (\ref{415}), (\ref{418}) and the definition of $ \mu_{\epsilon}(x) $, we obtain
\begin{align}
\label{eq15}
\lim_{\epsilon \rightarrow 0}\sum_{k}\int_{\partial B_{\epsilon}(x_{\epsilon,k})\cap D}(\gamma \nabla \omega_{\epsilon})\cdot (\gamma \overrightarrow{\nu})\phi v_{\epsilon} dS_{x} \ = \ & \lim_{\epsilon \rightarrow 0}\sum_{k}\int_{\partial B_{\epsilon}(x_{\epsilon,k})\cap D}(\gamma \nabla q_{\epsilon})\cdot (\gamma \overrightarrow{\nu})\phi v_{\epsilon} dS_{x}   \nonumber  \\
\geq \ & -\lim_{\epsilon \rightarrow 0}\int_{D}\phi v_{\epsilon}\mu_{\epsilon}(x)dx.
\end{align}
From our assumption, we have
$$  \lim_{\epsilon\rightarrow 0}\int_{D}\phi v_{\epsilon} \mu_{\epsilon}(x) dx \ = \ \int_{\Sigma}\phi v \mu(x)dS_{x}.   $$
Therefore
$$\lim_{\epsilon\rightarrow 0}\sum_{k}\int_{\partial B_{\epsilon}(x_{\epsilon,k})\cap D}(\gamma\nabla\omega_{\epsilon})\cdot(\gamma\overrightarrow{\nu})\phi v_{\epsilon} dS_{x}\ \ \ \geq \ \ \ -\int_{\Sigma}\phi v \mu(x)dS_{x} .$$

This finishes the proof of the lemma \ref{le42} .
\end{proof}

\section*{Acknowledgments}
The work of J. Li was partially supported by the NSF 
of China No. 11971221, the Shenzhen Sci-Tech Fund No. RCJC20200714114556020, JCYJ20190809150413261,   and JCYJ20170818153840322, 
and Guangdong Provincial Key Laboratory of Computational Science
and Material Design No. 2019B030301001.
The work of H. Liu was supported by the startup fund from City University of Hong Kong and the Hong Kong RGC General Research Fund (projects 12301420, 12302919 and 12301218). The work of L. Tang was partially supported by NNSFC grants of China (No.11831009) and the Fundamental Research Funds for the Central Universities (No. CCNU19TS032).
The work of J. Wang was partially supported by the Shenzhen Sci-Tech Fund No. JCYJ20180307151603959.

\end{document}